\documentclass{amsart}
\usepackage{amsmath,amssymb,amsthm, amsbsy}
\usepackage{bbm}
\usepackage[left=2cm, right=2cm, top=2cm, bottom=2cm]{geometry}
\theoremstyle{definition}
\newtheorem{defn}{Definition}[section]

\newtheorem{thm}[defn]{Theorem}
\newtheorem*{thm*}{Theorem}
\newtheorem{theorem}{Theorem}

\newtheorem{cor}[defn]{Corollary}
\newtheorem{prop}[defn]{Proposition}

\newtheorem{example}[defn]{Example}

\newcommand{\gap}{\, \vspace{1em}}

\newcommand{\disc}{\, \text{disc}}
\newcommand{\Gal}{\, \text{Gal}}
\newcommand{\sign}{\, \text{sign}}

\newcommand{\N}{\mathbb{N}}
\newcommand{\Z}{\mathbb{Z}}
\newcommand{\Q}{\mathbb{Q}}

\newcommand{\F}{\mathbb{F}}
\newcommand{\M}{\mathcal{M}}
\newcommand{\I}{\mathcal{I}}

\newcommand{\Frob}{\text{Frob}}
\newcommand{\p}{\mathfrak{p}}
\let\oldproofname=\proofname
\renewcommand{\proofname}{\rm\bf{\oldproofname}}
\renewcommand{\leq}{\leqslant}

\renewcommand{\subset}{\subseteq}

\newcommand\restr[2]{{
  \left.\kern-\nulldelimiterspace 
  #1 
  \vphantom{\big|} 
  \right|_{#2} 
  }}

\begin{document}

\pagenumbering{arabic}

\title{A ``Proto-Pellet's Formula" for the M\"{o}bius Function}
\author[A. Afshar]{Ardavan Afshar}
\address{Department of Mathematics\\University College London\\
25 Gordon Street, London, England}
\email{ardavan.afshar.15@ucl.ac.uk}

\maketitle

\begin{abstract}
We give a short proof of ``Pellet's Formula" for the M\"{o}bius Function on $\F_q[T]$, deriving an intermediate formula (which we call ``Proto-Pellet's Formula") along the way. We then construct and prove an analogous ``Proto-Pellet's Formula" for the M\"{o}bius Function for a number field (including the usual M\"{o}bius function on the integers).
\end{abstract}

\gap

\section{Introduction} 

\gap

Let $q$ be an odd prime power, and let $\M = \{f \in \F_q[T]: f \text{ monic}\}$. Define the M\"{o}bius Function on $\M$ by
$$ \mu(f) = \begin{cases} (-1)^r &\text{ if } f=p_1 \cdots p_r \text{ distinct primes} \\ 0 &\text{ else} \end{cases}$$
for $f \in \M$, in analogy with the usual M\"{o}bius Function on the integers. In $\M$ we have the following formula for $\mu$
\begin{theorem}[``Pellet's Formula"] \label{Pellet}
Let $f \in \M$, and let $ \chi $ be the quadratic character on $\F_q$. Then $$ \mu(f) = (-1)^{\deg f}\chi(\text{disc}f) $$ 
\end{theorem}
a proof of which can be found, for example, as Lemma 4.1 in \cite{Conrad1}. This is an important ingredient in the setting of $\F_q[T]$, such as in  
Sawin and Shusterman's proof of the Chowla and Twin Prime conjectures in $\F_q[T]$ (see \cite{S-S}). \\

We give another short proof of this formula, and in proving it, we derive an intermediate formula, which we call ``Proto-Pellet's Formula". Let $\Frob_q$ be the Frobenius element on ${\overline\F_q}$, which sends $\alpha$ to $\alpha^q$, then we have
\begin{theorem}[``Proto-Pellet's Formula"] \label{ProtoPellet}
Let $f \in \M$ square free, then $$ \mu(f) = (-1)^{\deg f} \text{sign}(\Frob_q | f) $$ where $\Frob_q | f$ denotes the action of $\Frob_q$ on the roots of $f$.
\end{theorem}

\gap

We then construct and prove an analogue of this formula for the M\"{o}bius Function for any number field $A/\Q$, which we define as follows. Let $\mathcal{O}_A$ be the ring of integers in $A$, and let 
$\mathcal{I}_A$ be the set of non-zero ideals in $\mathcal{O}_A$. Then the  M\"{o}bius Function for $A/\Q$, $\mu_A: \mathcal{I}_A \rightarrow \{-1, 0 , 1\}$ is defined by  
$$ \mu_A(I) = \begin{cases} (-1)^r &\text{ if } I=\mathfrak{p}_1 \cdots \mathfrak{p}_r \text{ distinct prime ideals} \\ 0 &\text{ else} \end{cases}$$
so that $\mu_\Q = \mu$ is the usual M\"{o}bius function on the integers. To do this, we first pick some arbitrary additive function $\nu: \I_A \to \N$ to mimic the role of degree in $\F_q[T]$, and for each prime ideal $\mathfrak{p}$ we construct a polynomial $f_{\mathfrak{p}, \nu}$ whose Galois group is isomorphic to $\Z/\nu(\mathfrak{p})\Z$ and generated by some homomorphism $\sigma_{\mathfrak{p}, \nu}$. We then lift these homomorphisms to an appropriate profinite Galois group and compose them to get a homomorphism $\sigma_{\nu}$ (which mimics the role of $\Frob_q$ in $\F_q[T]$), and set $f_{I, \nu} = \prod_{\mathfrak{p}|I \text{ prime}} f_{\mathfrak{p}, \nu} $ to prove
\begin{theorem}[``Proto-Pellet's Formula" for number fields] \label{ProtoPelletIntegers1}
Let $\nu: \I_A \to \N$ be an additive function. Then there exists a Galois homomorphism $\sigma_{\nu} \in \Gal(\overline{\Q}/\Q)$ and a family of polynomials $\left(f_{I, \nu}\right)_{I \in \mathcal{I}_A}$ such that, for all $I \in \I_A$ square-free
$$ \mu_A(I) = (-1)^{\nu(I)} \text{sign}(\sigma_\nu | f_{I, \nu}) $$
where $\sigma_\nu | f$ denotes the action of $\sigma_\nu$ on the roots of $f$.
\end{theorem}

\gap

\section{Pellet's Formula in $\F_q[T]$} \label{PelletSection}

\gap

We begin with a proof of Theorem \ref{ProtoPellet}:

\begin{proof}[Proof of Theorem \ref{ProtoPellet}]
For a square-free polynomial $f = p_1 \cdots p_r \in \M$, notice that $\Frob_q$ acts on the roots of $f$ as a product of cycles $\tau_1, \cdots, \tau_r$, where $\tau_i$ permutes the roots of $p_i$ as a cycle of length $\deg p_i$. This means that $\text{sign}(\Frob_q|f) = (-1)^{\deg p_1-1} \cdots(-1)^{\deg p_r-1} = (-1)^{\deg f - r} $.
\end{proof}

\gap

To move from this to Pellet's formula, we first define

\begin{defn}
For a polynomial $f$, let $\text{Aut}(f)$ be the automorphism group on the roots of $f$ and let $\text{Gal}(f)$ be its Galois closure. Let $\text{Sym}(f)$ be the symmetric group, and $\text{Alt}(f)$ the alternating group, on the roots of $f$. 
\end{defn}

and then prove an auxiliary result

\begin{prop} \label{Gal-disc}
Let $K$ be a field and $f \in K[T]$ square-free, and let $\rho_f : \text{Gal}(f) \hookrightarrow \text{Sym}(f)$ be the natural inclusion map.
Then $\rho_f(\text{Gal}{f}) \subset \text{Alt}(f) \iff \disc f \text{ is a square in } K^{\times}$.
\end{prop}
\begin{proof} Let $f$ have roots $\alpha_1, \cdots \alpha_n$ in $\overline{K}$ so that $ \disc f = \disc_K f = \prod_{i < j} (\alpha_i - \alpha_j)^2$. Then we have that
\begin{align*}
\disc f \text{ is a square in } K^{\times} &\iff \forall \sigma \in \Gal(f) \quad \sigma(\sqrt{\disc f}) = \sqrt{\disc f} \\
&\iff  \forall \sigma \in \Gal(f) \quad \sigma\left(\prod_{i < j} (\alpha_i - \alpha_j)\right) = \prod_{i < j} (\alpha_i - \alpha_j) \\
&\iff  \forall \sigma \in \Gal(f) \quad \prod_{i < j} (\alpha_i - \alpha_j) = \left(\prod_{i < j} \sigma(\alpha_i) - \sigma(\alpha_j)\right) = \text{sign}(\sigma)  \prod_{i < j} (\alpha_i - \alpha_j) \\
&\iff  \forall \sigma \in \Gal(f) \quad \text{sign}(\sigma) = 1.
\end{align*}
\end{proof}
\begin{cor} \label{frob-disc}
Let $f \in \M$ square-free, and let $ \chi $ be the quadratic character on $\F_q$. Then $$\text{sign}(\Frob_q | f) = \chi(\disc f)$$  
\end{cor}
\begin{proof}
For $f \in \M$ square-free, $\text{Gal}(f)$ is generated by $\Frob_q$, so by Proposition \ref{Gal-disc} we have that $$\disc f \text{ is a square in } \F_q^{\times} \iff \text{sign}(\Frob_q | f) = 1$$ from which the corollary follows.
\end{proof}

So, combining Theorem \ref{ProtoPellet} with Corollary \ref{frob-disc}, and noting that the discriminant vanishes on non-square-free polynomials, we get a proof of Theorem \ref{Pellet}.

\gap

\section{``Proto-Pellet's Formula" in number fields}

\gap

To begin, we pick some additive function $\nu: \mathcal{I}_A \to \N$ to mimic the role of degree in $\F_q[T]$.

\gap

Let $\mathcal{P}_A$ be the set of prime ideals in $A$, and put some order relation $<_A$ on $\mathcal{P}_A$: for example, one can order the prime ideals first by the size of their norm and then order those with the same norm arbitrarily (since there are only finitely many). For $\mathfrak{\p}$ a prime ideal, pick $q = q(\p, \nu)$ to be the minimal prime such that $q \equiv 1 \mod \nu(\p)$ and $q \neq q(\p', \nu)$ for some $\p' <_A \p$ prime (which  is always possible, by Dirichlet's Theorem for primes in arithmetic progression). Then consider $L = L_{\p, \nu} = \Q(\zeta_q)$ where $\zeta_q = e^{2\pi i/q}$ and note that $\Gal(L/\Q) \cong \Z/(q-1)\Z$. So, since $\nu(\p)$ divides $q-1$, there is some subgroup $H$ of $\Gal(L/\Q)$ such that $H \cong \Z/\frac{q-1}{\nu(\p)}\Z$. Therefore, if we let $K=K_{\p, \nu}=L^H:=\{x \in L \ | \ \sigma(x)=x \ \forall \sigma \in H\}$, we know by the Galois Correspondence that $\Gal(K/\Q) \cong \Z/\nu(\p)\Z$.

\gap

Moreover, by the Primitive Element Theorem, there exists some $\alpha_{\p, \nu} \in L$ such that $K = \Q(\alpha_{\p, \nu})$. So, if we let $f_{\p, \nu}$ be the minimal polynomial of $\alpha_{\p, \nu}$, then we have that $\Gal(f_{\p, \nu}) \cong \Z/\nu(\p)\Z$ and is generated by some $\sigma_{\p, \nu}$ which acts as a $\nu(\p)$-cycle on the roots of $f_{\p, \nu}$. In particular, this means that $\sign(\sigma_{\p, \nu} | f_{\p, \nu}) = (-1)^{\nu(\p)-1}$. 

\gap

Next, let $K_\nu = \Q(\{\alpha_{\p, \nu} \ |  \ \p \in \mathcal{P}_A\})$ and extend each $\sigma_{\p, \nu}$ to a map $\overline{\sigma}_{\p, \nu} \in \Gal(K_\nu/\Q)$ by $$\overline{\sigma}_{\p, \nu}(\alpha_{\p', \nu}) = \begin{cases} \sigma_{\p, \nu}(\alpha_{\p, \nu}) &\text{ if } \p = \p' \\ \alpha_{\p', \nu} &\text{ if } \p \neq \p' \end{cases} $$
The fact that we can do this follows from Proposition \ref{extension}. Then consider $\sigma_\nu \in \Gal(K_\nu/\Q)$ defined by $\sigma_\nu(\alpha_{\p, \nu}) = \sigma_{\p, \nu} (\alpha_{\p, \nu})$ for all prime ideals $\p \in \mathcal{P}_A$, so that by construction we have that $\sign(\sigma_{\nu} | f_{\p, \nu}) = \sign(\sigma_{\p, \nu} | f_{\p, \nu})$. We observe that $\sigma_\nu$ is well-defined by noting that it is a composition of the $\overline{\sigma}_{\p, \nu}$ over all prime ideals $\p \in \mathcal{P}_A$.

\gap

We take a brief aside to prove Proposition \ref{extension} as promised:

\begin{prop} \label{extension}
Let $\alpha_{\p, \nu}$ and $\sigma_{\p, \nu}$ be defined as above. Let $\{\p, \p_1, \cdots, \p_k\}$ be a set of distinct prime ideals in $\mathcal{P}_A$. Then we can extend $\sigma_{\p, \nu}$ to a map $\sigma^{(k)}_{\p, \nu} \in \Gal(\Q(\{\alpha_{\p', \nu} \ | \ \p' \in \{\p, \p_1, \cdots, \p_k\} \})/\Q)$ by $$\sigma^{(k)}_{\p, \nu}(\alpha_{\p', \nu}) = \begin{cases} \sigma_{\p, \nu}(\alpha_{\p, \nu}) &\text{ if } \p' = \p \\ \alpha_{\p', \nu} &\text{ if } \p' \in \{\p_1, \cdots, \p_k\} \end{cases} $$
\end{prop}
\begin{proof}
We proceed by induction and note that the case $k=0$ is trivially true. Suppose that the case $k=l$ is true, so that we have constructed our desired function $\sigma_{\p, \nu}^{(l)}$ on $K^{(l)}$, where $K^{(m)}:= \Q(\{\alpha_{\p', \nu} \ | \ \p' \in \{\p, \p_1, \cdots, \p_m\} \})$, and we seek to extend this to $\sigma_{\p, \nu}^{(l+1)}$ on $K^{(l+1)}$. Let $K = \Q(\alpha_{\p_{l+1}, \nu})$ so that $K^{(l+1)} = K^{(l)}K$, and so that, by Theorem 1.1 of \cite{Conrad2} there is an injective homomorphism
$$ \rho: \Gal(K^{(l+1)}/\Q) = \Gal(K^{(l)}K/\Q) \to \Gal(K^{(l)}/\Q) \times \Gal(K/\Q) $$
given by $ \rho(\sigma) = (\restr{\sigma}{K^{(l)}}, \restr{\sigma}{K})$. If we can show that $\rho$ is an isomorphism, then we can set $\sigma_{\p, \nu}^{(l+1)} = \rho^{-1}(\sigma_{\p, \nu}^{(l)}, \iota)$ (where $\iota$ is the identity function) and then we are done. \\

So suppose that $\rho$ is not an isomorphism, which means that $|\Gal(K^{(l+1)}/\Q)| < |\Gal(K^{(l)}/\Q)||\Gal(K/\Q)|$ and so $[K^{(l+1)}: \Q] < [K^{(l)}: \Q][K : \Q]$. Now define $L^{(m)}  = \Q(\{\zeta_{q(\p', \nu)} \ | \ \p' \in \{\p, \p_1, \cdots, \p_{m}\}\})$ (where $q(\p, \nu)$ is defined as above) and observe that, by our assumption
$$ [L^{(l+1)}: \Q] = [L^{(l+1)}: K^{(l+1)}][K^{(l+1)} : \Q] < [L^{(l+1)}: K^{(l+1)}][K^{(l)}: \Q][K : \Q] $$
and we can recursively compute
\begin{align*}
[L^{(l+1)}: K^{(l+1)}] &= [L^{(l)}(\zeta_{q(\p_{l+1}, \nu)}):L^{(l)}(\alpha_{\p_{l+1}, \nu})] [L^{(l)}(\alpha_{\p_{l+1}, \nu}):K^{(l+1)}]\\
&\leq [\Q(\zeta_{q(\p_{l+1}, \nu)}):\Q(\alpha_{\p_{l+1}, \nu})] [\Q(\alpha_{\p_{l+1}, \nu})(\zeta_{q(\p, \nu)}, \zeta_{q(\p_1, \nu)}, \cdots, \zeta_{q(\p_l, \nu)}):\Q(\alpha_{\p_{l+1}, \nu})(\alpha_{\p, \nu}, \alpha_{\p_1, \nu}, \cdots, \alpha_{\p_l, \nu})] \\
&\leq \frac{q(\p_{l+1}, \nu) -1}{\nu(\p_{l+1})}[\Q(\zeta_{q(\p, \nu)}, \zeta_{q(\p_1, \nu)}, \cdots, \zeta_{q(\p_l, \nu)}):\Q(\alpha_{\p, \nu}, \alpha_{\p_1, \nu}, \cdots, \alpha_{\p_l, \nu})] \\
&= \frac{q(\p_{l+1}, \nu)-1}{\nu(\p_{l+1})} [L^{(l)}: K^{(l)}] \leq \frac{q(\p, \nu)-1}{\nu(\p)}\prod_{j=1}^{l+1} \frac{q(\p_j, \nu)-1}{\nu(\p_j)}
\end{align*}
and so
\begin{align*}
[L^{(l+1)}: \Q] &< \left(\frac{q(\p, \nu)-1}{\nu(\p)}\prod_{j=1}^{l+1} \frac{q(\p_j, \nu)-1}{\nu(\p_j)}\right)[K^{(l)}: \Q][K : \Q]  \\
&\leq \left(\frac{q(\p, \nu)-1}{\nu(\p)}\prod_{j=1}^{l+1} \frac{q(\p_j, \nu)-1}{\nu(\p_j)}\right)\left(\nu(\p)\prod_{j=1}^{l} \nu(\p_j)\right)\left(v(\p_{l+1})\right) = (q(\p, \nu)-1) \prod_{j=1}^{l+1} (q(\p_j, \nu)-1)
\end{align*}
But, since $\{q(\p', \nu) \ | \ \p' \in \{\p, \p_1, \cdots, \p_{l+1}\}\}$ is a set of distinct primes, $\Q(\{\zeta_{q(\p', \nu)} \ | \ \p' \in \{\p, \p_1, \cdots, \p_{l+1}\}\}) = \Q(\zeta_{q(\p, \nu)q(\p_1, \nu) \cdots q(\p_{l+1}, \nu)})$ and so $[L^{(l+1)}: \Q] = (q(\p, \nu)-1) \prod_{j=1}^{l+1} (q(\p_j, \nu)-1)$, which is a contradiction.
\end{proof}

\gap

To put this all together, for $I \in \mathcal{I}_A$ we define $f_{I, \nu} = \prod_{\p \in \mathcal{P}_A: \p|I} f_{\p, \nu} $, where the product counts prime ideals with multiplicity. This allows us to formulate the following analogue of a ``Proto-Pellet's Formula" for the M\"obius function for $A$: 

\begin{thm} \label{ProtoPelletIntegers2}
Let $\nu: \mathcal{I}_A \to \N$ be an additive function, and $I \in \mathcal{I}_A$ square-free. Then for $\sigma_{\nu}$ and $f_{I, \nu}$ as defined above, we have:
$$ \mu(I) = (-1)^{\nu(I)} \text{sign}(\sigma_\nu | f_{I, \nu}) $$
\end{thm}
\begin{proof} For $I \in \I_A$ square-free we have that $I = \p_1 \cdots \p_r$ for some distinct prime ideals $\p_1, \cdots, \p_r \in \mathcal{P}_A$ and so
$$\sign(\sigma_{\nu} | f_{I, \nu}) = \prod_{\p \in \mathcal{P}_A: \p|I} \sign(\sigma_{\nu} | f_{\p, \nu}) = \prod_{\p \in \mathcal{P}_A: \p|I} (-1)^{\nu(\p)-1} = (-1)^{\nu(I) - r} = (-1)^{\nu(n)}\mu(I)$$
\end{proof}

Theorem \ref{ProtoPelletIntegers1} now follows as an immediate corollary.

\gap

Finally, we consider the trivial example as a check

\begin{example}
Let $\nu = \omega_A$ where $\omega_A(I) = \#\{\mathfrak{p} \in \mathcal{P}_A \text{ distinct} : \mathfrak{p}|I\}$
so that $\omega_A(\p) = 1$ for all prime ideals $\p \in \mathcal{P}_A$. Then for each $\p$, $\Gal(f_{\p, \omega_A}) \cong \Z/\omega_A(\p)\Z$ is trivial and $\sigma_{\p, \omega_A}$ is just the identity map, which means that $\text{sign}(\sigma_{\p, \omega_A} | f_{\p, \omega_A}) = 1$. Therefore, for all square-free $I \in \mathcal{I}_A$, we have that $\text{sign}(\sigma_{\omega_A} | f_{I, \omega_A}) = 1$, which by Theorem  \ref{ProtoPelletIntegers2} implies that $\mu_A(I) = (-1)^{\omega_A(I)}$, recovering the definition of $\mu_A$.
\end{example}

\gap

\section*{Acknowledgements}

The author would like to thank Andrew Granville for his guidance, and Richard Hill for useful discussions. The research leading to these results has received funding from the European Research Council under the European Union's Seventh Framework Programme (FP7/2007-2013), ERC grant agreement n$^{\text{o}}$ 670239.

\gap

\end{document}